\documentclass[10pt]{article}
\usepackage{fancyhdr,a4wide}
\usepackage[mathscr]{eucal}
\usepackage{amsmath}
\usepackage{amssymb}
\usepackage{amsthm}
\usepackage{float}
\usepackage{caption}
\usepackage[caption = false]{subfig}
\usepackage{enumerate}
\usepackage{fullpage}
\usepackage{graphicx}
\usepackage{multicol}
\usepackage{tikz,  mathtools, soul, mathdots, yhmath}
\usetikzlibrary{calc,decorations.markings,arrows}

\usepackage{hyperref}

\newcommand{\id}{\operatorname{id}\nolimits}

\input xy
\xyoption{all}

\usepackage[all]{xy}

\begin{document}
\title{Decomposable Extensions Between  Rank $1$  Modules in Grassmannian Cluster Categories} 

\author{Dusko Bogdanic and Ivan-Vanja Boroja}

\date{}

\def\comment#1{\textcolor[rgb]{0.69,0.0,0.0}{ \texttt{#1}}}  %
\def\vorl#1{\textcolor[rgb]{0.0,0.2,0.8}{[Lecture #1]}}

\definecolor{light-grey}{gray}{0.6}  



\newtheorem{lm}{Lemma}[section]
\newtheorem{prop}[lm]{Proposition}
\newtheorem{satz}[lm]{Satz}

\newtheorem{corollary}[lm]{Corollary}
\newtheorem{cor}[lm]{Korollar}
\newtheorem{claim}[lm]{Claim}
\newtheorem{theorem}[lm]{Theorem}
\newtheorem*{thm}{Theorem}

\theoremstyle{definition}
\newtheorem{defn}[lm]{Definition}
\newtheorem{qu}{Question}
\newtheorem{ex}[lm]{Example}
\newtheorem{exas}[lm]{Examples}
\newtheorem{exc}[lm]{Exercise}
\newtheorem*{facts}{Facts}
\newtheorem{rem}[lm]{Remark}

\theoremstyle{remark}

\newcommand{\perm}{\operatorname{Perm}\nolimits}
\newcommand{\NN}{\operatorname{\mathbb{N}}\nolimits}
\newcommand{\ZZ}{\operatorname{\mathbb{Z}}\nolimits}
\newcommand{\QQ}{\operatorname{\mathbb{Q}}\nolimits}
\newcommand{\RR}{\operatorname{\mathbb{R}}\nolimits}
\newcommand{\CC}{\operatorname{\mathbb{C}}\nolimits}
\newcommand{\PP}{\operatorname{\mathbb{P}}\nolimits}
\newcommand{\cA}{\operatorname{\mathcal{A}}\nolimits}
\newcommand{\LL}{\operatorname{\Lambda}\nolimits}

\newcommand{\MM}{\operatorname{\mathbb{M}}\nolimits}
\newcommand{\Fk}{\mathcal{F}_{k,n}}
\newcommand{\Mk}{M_{k,n}}

\newcommand{\ad}{\operatorname{ad}\nolimits}
\newcommand{\im}{\operatorname{im}\nolimits}
\newcommand{\Char}{\operatorname{char}\nolimits}
\newcommand{\Aut}{\operatorname{Aut}\nolimits}

\newcommand{\Id}{\operatorname{Id}\nolimits}
\newcommand{\ord}{\operatorname{ord}\nolimits}
\newcommand{\ggT}{\operatorname{ggT}\nolimits}
\newcommand{\lcm}{\operatorname{lcm}\nolimits}

\newcommand{\Gr}{\operatorname{Gr}\nolimits}

\maketitle

\begin{abstract} Rank $1$ modules are the building blocks of the category  ${\rm CM}(B_{k,n}) $ of Cohen-Macaulay modules over a quotient $B_{k,n}$ of a preprojective algebra of affine type $A$. Jensen, King and Su showed in \cite{JKS16} that the category ${\rm CM}(B_{k,n})$ provides an additive categorification of  the cluster algebra structure on the coordinate ring $\mathbb C[{\rm Gr}(k, n)]$ of the Grassmannian variety of $k$-dimensional subspaces in $\mathbb C^n$. Rank $1$ modules are indecomposable, they are known to be in bijection with $k$-subsets of $\{1,2,\dots,n\}$, and their explicit construction has been given in \cite{JKS16}.  In this paper, we give necessary and sufficient conditions for indecomposability of an arbitrary rank 2 module in ${\rm CM}(B_{k,n})$ whose filtration layers are tightly interlacing.  We give an explicit construction of all rank 2 decomposable modules that  appear as extensions between rank 1 modules corresponding to tightly interlacing $k$-subsets $I$ and $J$. 
\end{abstract}
\noindent



\section{Introduction} 
%

A categorification of the cluster algebra structure on the homogeneous coordinate ring $\mathbb C[{\rm Gr}(k, n)]$ of the Grassmannian variety of $k$-dimensional subspaces in $\mathbb C^n$ has been given by Geiss, Leclerc, and Schroer \cite{GLS06, GLS08} in terms of a subcategory of the category of finite dimensional modules over the preprojective algebra of type $A_{n-1}$. Jensen, King, and Su \cite{JKS16} gave a new categorification of this cluster structure using the maximal Cohen-Macaulay modules over the completion of an algebra $B_{k,n}$ which is a quotient of the preprojective algebra of type $A_{n-1}$.  Rank $1$ modules are the building blocks of the category  ${\rm CM}(B_{k,n}) $ of Cohen-Macaulay modules over a quotient $B_{k,n}$ of a preprojective algebra of affine type $A_{n-1}$. Rank 1 modules are indecomposable, they are known to be in bijection with $k$-subsets of $[n]=\{1,2,\dots,n\}$, and their explicit construction has been given in \cite{JKS16}. These are the building blocks of the category as any module in ${\rm CM}(B_{k,n}) $ can be filtered by rank $1$ modules (the filtration is noted in the profile of a module, \cite[Corollary 6.7]{JKS16}). The number of rank 1 modules appearing in the filtration of a given module is called the rank of that module. In \cite{BBL}, we explicitly constructed all indecomposable rank 2 modules in tame cases. 

In this paper, we give necessary and sufficient conditions for indecomposability of an arbitrary rank 2 module in ${\rm CM}(B_{k,n})$ whose filtration layers are tightly interlacing. Moreover, we construct explicitly all rank 2 decomposable Cohen-Macaulay $B_{k,n}$-modules that appear as middle terms in the short exact sequences where the end terms are rank 1 modules corresponding to tightly interlacing subsets.  The central combinatorial notion throughout this paper is that of $r$-interlacing (Definition~\ref{interlacing}). If $I$ and $J$ are $k$-subsets of $\{1,\ldots, n\}$, then
$I$ and $J$ are said to be {\em $r$-interlacing} if there exist subsets 
$\{i_1,i_3,\dots,i_{2r-1}\}\subset I\setminus J$ and $\{i_2,i_4,\dots, i_{2r}\}\subset J\setminus I$ 
such that $i_1<i_2<i_3<\dots <i_{2r}<i_1$ (cyclically) 
and if there exist no larger subsets of $I$ and of $J$ with this property.

Denote by $L_I$ the rank 1 indecomposable module corresponding to the $k$-subset $I$. By \cite[Proposition 5.6]{JKS16},  ${\rm Ext}_{B}^1(L_I,L_J)\neq 0$ if and only if $I$ and $J$ are  $r$-interlacing, where $r\geq 2$. In particular, rank 1 modules are rigid, i.e.\  ${\rm Ext}_{B}^1(L_I,L_I)=0$ for every $I$. 
This means that if the sets $I$ and $J$ are $1$-interlacing, then the only module appearing as the middle term in short exact sequences with end terms $L_I$ and $L_J$ is the direct sum $L_I\oplus L_J$. For this reason, we will assume most of the time that $I $ and $J$ are $r$-interlacing with $r\geq 2$.  Note also that, by Theorem 3.7 in \cite{BB}, ${\rm Ext}_{B}^1(L_I,L_J)\cong {\rm Ext}_{B}^1(L_J,L_I)$, so we have the same arguments for the short exact sequences with $L_I$ as the left term  and $L_J$ as the right term, and for the short exact sequences with $L_J$ as the left term  and $L_I$ as the right term.

The paper is organized as follows. In Section 2, we recall the definitions and key results
about Grassmannian cluster categories. In Section 3, we study the filtration $I\mid J$, where  $I=\{1,3,\dots, 2r-1\}$ and  $J=\{2,4, \ldots, 2r\},$ in the case $(r,2r)$. We explain how the general case of a module with tight $r$-interlacing filtration layers reduces to the case of  the module with  filtration      $I\mid J$. For the filtration layers $I$ and $J$ of a module with profile $I\mid J$, we construct all decomposable rank 2 modules that are extensions of these rank 1 modules, i.e.\ we construct all decomposable modules that appear as middle terms in short exact sequences with $L_I$ and $L_J$ as end terms.  In particular, we associate with every subset of peaks of the rim $I$  a decomposable rank 2 module that is extension  of $L_{J}$ by $L_{I}$.

Our main results are Theorem \ref{t6} 
in which  we give  necessary and sufficient conditions for a rank 2 module with filtration $I\mid J$ to be indecomposable, and Theorem \ref{paths} 
in which we give an explicit construction of all rank 2 decomposable modules that  appear as extensions between rank 1 modules corresponding to the subsets $I$ and $J$.

%

\section{Preliminaries}
We follow closely the exposition from \cite{JKS16, BB, BBGE, BBL} in order to introduce notation and background results. Let $\Gamma_n$ be the quiver of the boundary algebra, with vertices $1,2,\dots, n$ 
on a cycle and arrows $x_i: i-1\to i$, $y_i:i\to i-1$ (see Figure \ref{quiv}). 
We write ${\rm CM}(B_{k,n})$ for the category of maximal Cohen-Macaulay modules for  
the completed path algebra $B_{k,n}$ of $\Gamma_n$, with relations 
$xy-yx$ and $x^k-y^{n-k}$ (at every vertex). The centre of $B_{k,n}$ is 
$Z:=\CC[|t|]$, where $t=\sum_ix_iy_i$.  
\begin{center}
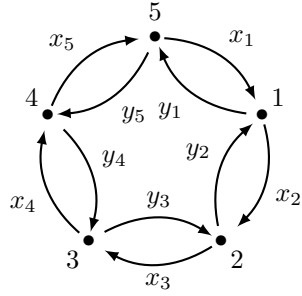
\begin{figure}[H]
\begin{center}
\begin{tikzpicture}[scale=1]
\newcommand{\radius}{1.5cm}
\foreach \j in {1,...,5}{
  \path (90-72*\j:\radius) node[black] (w\j) {$\bullet$};
  \path (162-72*\j:\radius) node[black] (v\j) {};
  \path[->,>=latex] (v\j) edge[black,bend left=30,thick] node[black,auto] {$x_{\j}$} (w\j);
  \path[->,>=latex] (w\j) edge[black,bend left=30,thick] node[black,auto] {$y_{\j}$}(v\j);
}
\draw (90:\radius) node[above=3pt] {$5$};
\draw (162:\radius) node[above left] {$4$};
\draw (234:\radius) node[below left] {$3$};
\draw (306:\radius) node[below right] {$2$};
\draw (18:\radius) node[above right] {$1$};
\end{tikzpicture}
\end{center}
\caption{The quiver $\Gamma_n$ for $n=5$.}\label{quiv} 
\end{figure}
\end{center}
The algebra $B_{k,n}$  coincides with the quotient of the completed path
algebra of the graph $C$ (a circular graph with vertices
$C_0=\mathbb Z_n$ set clockwise around a circle, and with the set of edges, $C_1$, also
labeled by $\mathbb Z_n$, with edge $i$ joining vertices $i-1$ and $i$), i.e.\ the doubled quiver as above,
by the closure of the ideal generated by the relations above (we view the completed path
algebra of the graph $C$ 
as a topological algebra via the $m$-adic topology, where $m$ is the two-sided ideal
generated by the arrows of the quiver, see \cite[Section 1]{DWZ08}). The algebra 
$B_{k,n}$, that we will often denote by $B$ when there is no ambiguity, 
was introduced in \cite[Section 3]{JKS16}. 
Observe that $B_{k,n}$ is isomorphic to $B_{n-k,n}$, so we will always assume that  $k\le \frac n 2$. 

\smallskip

The (maximal) Cohen-Macaulay $B$-modules are precisely those which are
free as $Z$-modules. Such a module $M$ is given by a representation
$\{M_i\,:\,i\in C_0\}$ of
the quiver with each $M_i$ a free $Z$-module of the same rank
(which is the rank of $M$).

\begin{defn}[\cite{JKS16}, Definition 3.5]
For any $B_{k,n}$-module $M$ and $K$ the field of fractions of $Z$, the {\bf rank} 
of $M$, denoted by ${\rm rk}(M)$,  is defined 
to be ${\rm rk}(M) = {\rm len}(M \otimes_Z K)$. 
\end{defn}

Note that $B\otimes_Z K\cong M_n ( K)$, 
which is a simple algebra. It is easy to check that the rank is additive on short exact sequences,
that ${\rm rk} (M) = 0$ for any finite-dimensional $B$-module 
(because these are torsion over $Z$) and 
that, for any Cohen-Macaulay $B$-module $M$ and every idempotent $e_j$, $1\leq j\leq n$, ${\rm rk}_Z(e_j M) = {\rm rk}(M)$, so that, in particular, ${\rm rk}_Z(M) = n  {\rm rk}(M)$.

\begin{defn}[\cite{JKS16}, Definition 5.1] \label{d:moduleMI}
For any $k$-subset   $I$  of $C_1$, we define a rank $1$ $B$-module
\[
  L_I = (U_i,\ i\in C_0 \,;\, x_i,y_i,\, i\in C_1)
\]
as follows.
For each vertex $i\in C_0$, set $U_i=\mathbb C[[t]]$,
for each edge $i\in C_1$, set
\begin{itemize}
\item[] $x_i\colon U_{i-1}\to U_{i}$ to be multiplication by $1$ if $i\in I$, and by $t$ if $i\not\in I$,
\item[] $y_i\colon U_{i}\to U_{i-1}$ to be multiplication by $t$ if $i\in I$, and by $1$ if $i\not\in I$.
\end{itemize}
\end{defn}

The module $L_I$ can be represented by a lattice diagram
$\mathcal{L}_I$ in which $U_0,U_1,U_2,\ldots, U_n$ are represented by columns of vertices (dots) from
left to right (with $U_0$ and $U_n$ to be identified), going down infinitely. 
The vertices in each column correspond to the natural monomial 
$\mathbb C$-basis of $\mathbb C[t]$.
The column corresponding to $U_{i+1}$ is displaced half a step vertically
downwards (respectively, upwards) in relation to $U_i$ if $i+1\in I$
(respectively, $i+1\not \in I$), and the actions of $x_i$ and $y_i$ are
shown as diagonal arrows. Note that the $k$-subset $I$ can then be read off as
the set of labels on the arrows pointing down to the right which are exposed
to the top of the diagram. For example, the lattice diagram $\mathcal{L}_{\{1,4,5\}}$
in the case $k=3$, $n=8$, is shown in Figure \ref{Lattice}.   
\begin{center}
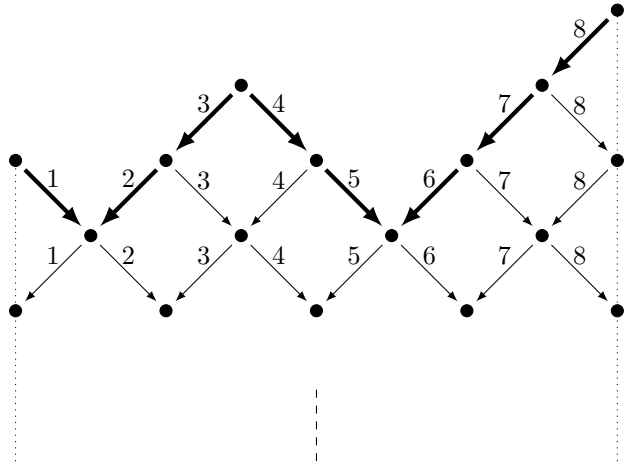
\begin{figure}[H]
\center
\begin{tikzpicture}[scale=1,baseline=(bb.base),
quivarrow/.style={black, -latex, thin}]
\newcommand{\seventh}{51.4} 
\newcommand{\circradius}{1.5cm}
\newcommand{\inradius}{1.2cm}
\newcommand{\outradius}{1.8cm}
\newcommand{\dotrad}{0.1cm} 
\newcommand{\bdrydotrad}{{0.8*\dotrad}} 
\path (0,0) node (bb) {}; 


\draw (0,0) circle(\bdrydotrad) [fill=black];
\draw (0,2) circle(\bdrydotrad) [fill=black];
\draw (1,1) circle(\bdrydotrad) [fill=black];
\draw (2,0) circle(\bdrydotrad) [fill=black];
\draw (2,2) circle(\bdrydotrad) [fill=black];
\draw (3,1) circle(\bdrydotrad) [fill=black];
\draw (3,3) circle(\bdrydotrad) [fill=black];
\draw (4,0) circle(\bdrydotrad) [fill=black];
\draw (4,2) circle(\bdrydotrad) [fill=black];
\draw (5,1) circle(\bdrydotrad) [fill=black];
\draw (6,0) circle(\bdrydotrad) [fill=black];
\draw (6,2) circle(\bdrydotrad) [fill=black];
\draw (7,1) circle(\bdrydotrad) [fill=black];
\draw (7,3) circle(\bdrydotrad) [fill=black];
\draw (8,2) circle(\bdrydotrad) [fill=black];
\draw (8,4) circle(\bdrydotrad) [fill=black];
\draw (8,0) circle(\bdrydotrad) [fill=black];


\draw [quivarrow,shorten <=5pt, shorten >=5pt, ultra thick] (0,2)-- node[above]{$1$} (1,1);
\draw [quivarrow,shorten <=5pt, shorten >=5pt] (1,1) -- node[above]{$1$} (0,0);
\draw [quivarrow,shorten <=5pt, shorten >=5pt, ultra thick] (2,2) -- node[above]{$2$} (1,1);
\draw [quivarrow,shorten <=5pt, shorten >=5pt] (1,1) -- node[above]{$2$} (2,0);
\draw [quivarrow,shorten <=5pt, shorten >=5pt, ultra thick] (3,3) -- node[above]{$3$} (2,2);
\draw [quivarrow,shorten <=5pt, shorten >=5pt] (2,2) -- node[above]{$3$} (3,1);
\draw [quivarrow,shorten <=5pt, shorten >=5pt] (3,1) -- node[above]{$3$} (2,0);
\draw [quivarrow,shorten <=5pt, shorten >=5pt, ultra thick] (3,3) -- node[above]{$4$} (4,2);
\draw [quivarrow,shorten <=5pt, shorten >=5pt] (4,2) -- node[above]{$4$} (3,1);
\draw [quivarrow,shorten <=5pt, shorten >=5pt] (3,1) -- node[above]{$4$} (4,0);
\draw [quivarrow,shorten <=5pt, shorten >=5pt, ultra thick] (4,2) -- node[above]{$5$} (5,1);
\draw [quivarrow,shorten <=5pt, shorten >=5pt] (5,1) -- node[above]{$5$} (4,0);
\draw [quivarrow,shorten <=5pt, shorten >=5pt, ultra thick] (6,2) -- node[above]{$6$} (5,1);
\draw [quivarrow,shorten <=5pt, shorten >=5pt] (5,1) -- node[above]{$6$} (6,0);
\draw [quivarrow,shorten <=5pt, shorten >=5pt] (6,2) -- node[above]{$7$} (7,1);
\draw [quivarrow,shorten <=5pt, shorten >=5pt] (7,1) -- node[above]{$7$} (6,0);
\draw [quivarrow,shorten <=5pt, shorten >=5pt, ultra thick] (7,3) -- node[above]{$7$} (6,2);
\draw [quivarrow,shorten <=5pt, shorten >=5pt] (7,3) -- node[above]{$8$} (8,2);
\draw [quivarrow,shorten <=5pt, shorten >=5pt] (8,2) -- node[above]{$8$} (7,1);
\draw [quivarrow,shorten <=5pt, shorten >=5pt, ultra thick] (8,4) -- node[above]{$8$} (7,3);
\draw [quivarrow,shorten <=5pt, shorten >=5pt] (7,1) -- node[above]{$8$} (8,0);

\draw [dotted] (0,-2) -- (0,2);
\draw [dotted] (8,-2) -- (8,4);

\draw [dashed] (4,-2) -- (4,-1);
\end{tikzpicture}
\caption{Lattice diagram of the module $L_{\{1,4,5\}}$} \label{Lattice}
\end{figure}
\end{center}

We see from Figure \ref{Lattice} that the module $L_I$ is determined by its upper boundary, denoted  by the thick lines, 
which we refer to as the {\em rim} of the module $L_I$ (this is why we call the $k$-subset $I$  the rim of $L_I$). 
Throughout this paper we will identify a rank 1 module $L_I$ with its rim. Moreover, most of the time we will omit the arrows in the rim of $L_I$ and represent it as an undirected graph. 

We say that $i$ is a {\em peak} of the rim $I$ if $i\notin I$ and $i+1\in I$.  In the above example, the peaks of $I=\{1,4,5\}$ are $3$ and $8$. We say that $i$ is a {\em valley} of the rim $I$ if $i\in I$ and $i+1\notin I$.  In the above example, the valleys of $I=\{1,4,5\}$ are $1$ and $5$.

\begin{prop}[\cite{JKS16}, Proposition 5.2]
Every rank $1$ Cohen-Macaulay $B_{k,n}$-module is isomorphic to $L_I$
for some unique $k$-subset $I$ of $C_1$.
\end{prop}

Every $B$-module has a canonical endomorphism given by multiplication by $t\in Z$.
For ${L}_I$ this corresponds to shifting $\mathcal{L}_I$ one step downwards.
Since $Z$ is central, ${\rm Hom}_B(M,N)$ is
a $Z$-module for arbitrary $B$-modules $M$ and $N$. If $M,N$ are free $Z$-modules, then so is ${\rm Hom}_B(M,N)$. In particular, for any two rank 1 
Cohen-Macaulay $B$-modules $L_I$ and $L_J$, ${\rm Hom}_B(L_I,L_J)$ is a free 
module of rank 1
over $Z=\mathbb C[[t]]$, generated by the canonical map given by placing the 
lattice of $L_I$ inside the lattice of $L_J$ as far up as possible so that no part of the rim of $L_I$ is strictly above the rim of $L_J$ \cite[Section 6]{JKS16}.

\begin{defn}[$r$-interlacing] \label{interlacing}
Let $I$ and $J$ be two $k$-subsets of $\{1,\dots,n\}$. The sets $I$ and $J$ are said to be {\em $r$-interlacing} if there exist subsets  \\ $\{i_1,i_3,\dots,i_{2r-1}\}\subset I\setminus J$ and $\{i_2,i_4,\dots, i_{2r}\}\subset J\setminus I$ 
such that $i_1<i_2<i_3<\dots <i_{2r}<i_1$ (cyclically) 
and if there exist no larger subsets of $I$ and of $J$ with this property.  We say that $I$ and $J$ are 
{\em tightly $r$-interlacing} if they are $r$-interlacing and $|I\cap J|=k-r.$
\end{defn}

\begin{defn}\label{def-rigid}
A $B$-module is \emph{rigid} if ${\rm Ext}^1_B (M,M)=0$.
\end{defn}

If $I$ and $J$ are $r$-interlacing $k$-subsets, where $r<2$, then ${\rm Ext}_{B}^1(L_I,L_J)=0$, in particular, 
rank 1 modules are rigid (see \cite[Proposition 5.6]{JKS16}).

Every indecomposable $M$ of rank $n$ in ${\rm CM}(B)$ has a filtration with  factors 
$L_{I_1},L_{I_{2}},\dots, L_{I_n}$ of rank 1. 
This filtration is noted in its \emph{profile}, 
${\rm pr} (M) = I_1 \mid I_2\mid\ldots \mid I_n$, \cite[Corollary 6.7]{JKS16}.
In the case of  a rank $2$ module $M$ with filtration $L_I\mid L_J$ (i.e.~with profile $I\mid J$), 
we picture this module by drawing the rim $J$ below the rim $I$, in such a way that $J$ is placed as far up as possible so that no part of the rim $J$ is strictly above the rim $I$. We refer to this picture of $M$ 
as its {\em lattice diagram.}
Note that there is at least one point where the rims $I$ and $J$ meet (see Figure \ref{lat}).  
\begin{figure}[H]
\begin{center}
{\includegraphics[width = 8cm]{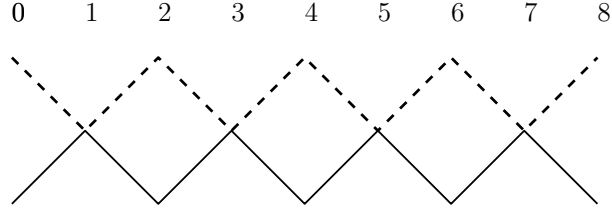}}
\caption{The lattice diagram of a module with filtration  $ L_{\{1, 3, 5,7\}}\mid L_{\{2,4, 6,8\}}$.} \label{lat}
\label{some example}
\end{center}
\end{figure}
The two rims in the lattice diagram of a rank 2 module $M$ form a number of regions between 
the points where the two rims meet but differ in direction before and/or after meeting. 
We call these regions the 
{\em boxes} formed by the rims or by the profile. 
The term box is a combinatorial tool which is very useful in finding conditions for indecomposability.  
However, let us point out that the module $M$ might be a direct sum in which case the lattice diagram is really a pair of  lattice diagrams of rank 1 modules. We still view the corresponding diagram as forming boxes. 
If $I$ and $J$ are $r$-interlacing, then they form exactly $r$-boxes if and only if they are tightly $r$-interlacing.
A lattice diagram with three boxes is shown in Figure~\ref{fig:boxes-poset}. Moreover, the filtration layers of a module $M$ give a poset structure. If $M$ is a rank 2 module with $r_1$ boxes, with $r_1\le r$, the poset structure associated with $M$ is $1^{r_1}\mid 2$, see Figure~\ref{fig:boxes-poset}. The poset consists of a tree with one vertex of degree $r_1$ and $r_1$ leaves, it has dimension $1$ at the leaves and dimension 2 at central vertex (we also refer to this as a {\it dimension lattice}). For background on the poset associated with an indecomposable module or to its profile, 
we refer to~\cite[Section 6]{JKS16} and \cite[Section 2]{BBGEL20}. 
\begin{figure}[H] 
\begin{center}
{\includegraphics[width = 9cm]{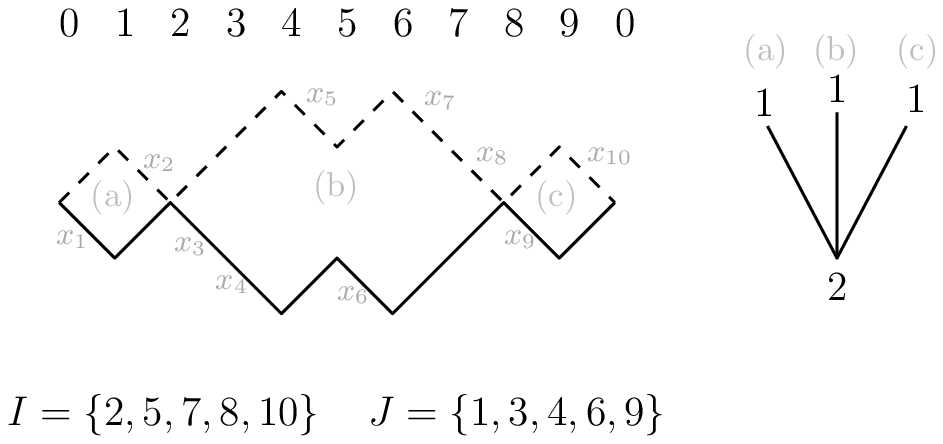}} 
\caption{The profile of  a  module  with $4$-interlacing layers forming three boxes  with poset $1^3\mid 2$. 
The dashed line shows the rim of $L_I$ with arrows $x_i$, $i\in I$, indicated. 
The solid line below is the rim of $L_J$, with arrows $x_i$, $i\in J$, indicated.} \label{fig:boxes-poset}
\end{center}
\end{figure}
A partial answer to the question of indecomposability of a rank 2 module in terms of its poset is given in the following proposition. 

\begin{prop}[\cite{BBGE}, Remark 3.2] \label{poset}
Let $M\in {\rm CM}(B_{k,n})$ be an indecomposable module with profile $I\mid J$. Then  $I$ and $J$ are  $r$-interlacing and their poset is $1^{r_1}\mid 2$, where $r\geq r_1\geq 3$.   
\end{prop}

This proposition tells us that when dealing with rank 2 indecomposable modules, we can assume that the poset of such a module  is of the form $1^{r_1}\mid 2$, for $r_1\geq 3$. 

Throughout the paper, our strategy to prove that a module is indecomposable is to show that its endomorphism ring does not have 
non-trivial idempotent elements. 
When we deal with a decomposable rank 2 module, in order to determine the summands of this module, we construct a 
non-trivial idempotent in its endomorphism ring, and then find corresponding eigenvectors at each vertex of the quiver and check 
the action of the morphisms $x_i$ on these eigenvectors.

\section{Tight $r$-interlacing}
In this section we construct all rank 2 decomposable modules with filtration $I\mid J$ in the case when $I$ and $J$ are  tightly $r$-interlacing $k$-subsets, i.e., when $|I\setminus J|=|J\setminus I|=r$ 
and non-common elements of $I$ and $J$ interlace, that is, $|I\cap J|=k-r$.  

We are interested in the modules $M$ that are decomposable and appear as the middle term in a short exact sequence of the form: 
$$0\longrightarrow L_J\longrightarrow M \longrightarrow L_I\longrightarrow 0.$$

In \cite{BBL},  we defined a rank 2 module $\MM(I,J)$ with filtration $L_I\mid L_J$ in a similar way as rank 1 modules 
are defined in ${\rm CM}(B_{k,n})$. We recall the construction here.   
Let $V_i:=\CC[|t|]\oplus\CC[|t|]$, $i=1,\dots, n$. 
The module $\MM(I,J)$ has $V_i$ at each vertex $1,2,\dots, n$ of $\Gamma_n$.   In order to have a module structure, for every $i$ we need to define  $x_i\colon V_{i-1}\to V_{i}$ and $y_i\colon V_{i}\to V_{i-1}$ in such a way that $x_iy_i=t\cdot \id$ and $x^k=y^{n-k}$.   

Since $L_J$ is a submodule of a rank 2 module $\MM(I,J)$, and $L_I$ is the quotient, if we extend the basis of $L_J$ to the basis of the module $\MM(I,J)$, then with respect to that basis all the matrices $x_i$, $y_i$ must be upper triangular with diagonal entries from the set $\{1,t\}$. More precisely, the diagonal of $x_i$ (resp.\ $y_i$) is $(1,t)$ (resp.\ $(t,1)$) if $i\in J\setminus I$, it is $(t,1)$ (resp.\ $(1,t)$)  if $i\in I\setminus J$, $(t,t)$ (resp.\ $(1,1)$) if $i\in I^c\cap J^c$, and $(1,1)$ (resp.\ $(t,t)$) if $i\in I\cap J$. The only entries in all these matrices that are left to be determined are the ones in the upper right corner. 

Let us assume that we deal with the profile $\{1,3,\dots, 2r-1\}\mid \{2,4,\ldots, 2r\}$  in the case $(r,2r)$. In the general case, all arguments are the same.  Denote by $b_i$ the upper right corner element of $x_i$. From  $x_iy_i=t\cdot id$, we have that the upper right corner element of $y_i$ is $-b_i$.  From the relation $x^k=y^{n-k}$ it follows that $\sum_{i=1}^{2r}b_i=0$. If $n=6$, $I=\{1,3,5\}$ and $J=\{2,4,6\}$, then our module $\MM(I,J)$ is 
\begin{figure}[h]
{\small \begin{center} \hfil
\xymatrix@C=5.5em{ 
V_0\ar@<.8ex>[r]^{\begin{pmatrix} t & b_1 \\ 0 & 1 \end{pmatrix}}
&V_1 \ar@<.8ex>[r]^{\begin{pmatrix} 1 & b_2 \\ 0 & t \end{pmatrix}} \ar@<.8ex>[l]^{\begin{pmatrix} 1 & -b_1 \\ 0 & t \end{pmatrix}} 
& V_2  \ar@<.8ex>[r]^{\begin{pmatrix} t & b_3 \\ 0 & 1 \end{pmatrix}}\ar@<.8ex>[l]^{\begin{pmatrix} t & -b_2 \\ 0 & 1 \end{pmatrix}} 
& V_3  \ar@<.8ex>[r]^{\begin{pmatrix} 1 & b_4 \\ 0 & t \end{pmatrix}} \ar@<.8ex>[l]^{\begin{pmatrix} 1 & -b_3 \\ 0 & t \end{pmatrix}}
&V_4 \ar@<.8ex>[r]^{\begin{pmatrix} t & b_5 \\ 0 & 1 \end{pmatrix}} \ar@<.8ex>[l]^{\begin{pmatrix} t & -b_4 \\ 0 & 1 \end{pmatrix}}
&V_5 \ar@<.8ex>[r]^{\begin{pmatrix} 1 & b_6 \\ 0 & t \end{pmatrix}} \ar@<.8ex>[l]^{\begin{pmatrix} 1 & -b_5 \\ 0 & t \end{pmatrix}}
&V_0\,\,\,\ar@<.8ex>[l]^{\begin{pmatrix} t & -b_6 \\ 0 & 1 \end{pmatrix}}
} \hfil
\end{center} }
\caption{A module with filtration $\{1,3,5\}\mid \{2,4,6\}$.}
\end{figure}

The question is how to determine the $b_i$'s so that the module $\mathbb M(I,J)$ is decomposable. In \cite{BBL}, we dealt with the tame cases $(3,9)$ and $(4,8)$, and more generally, the  $3$-interlacing case, and we constructed all such modules and given criteria, in terms of divisibility by $t$ of the sums $b_i+b_{i+1}$ (where $i$ is odd), for the constructed module to be indecomposable. Moreover, in the case of a decomposable module, we determined the summands of such  a module.  In this paper we construct all decomposable modules in the general case of tight $r$-interlacing. We first consider the case $(r,2r)$ and show  how the general case reduces to this case.

Assume first that $\mathbb M(I,J)$ is decomposable and that $L_J$ is a direct summand of $\mathbb M(I,J)$. Then there exists a retraction $\mu=(\mu_i)_{i=1}^n$ such that $\mu_i\circ \theta_i=id$, where $(\theta_i)_{i=1}^n$ is the natural injection of $L_J$ into $\mathbb M(I,J)$. Using the same basis as before, we can assume that $\mu_i=[1\,\, \alpha_i ]$. From the commutativity relations we have $id\circ \mu_i=\mu_{i+1}\circ x_{i+1}$ for $i$ odd, and  $t\cdot id\circ \mu_i=\mu_{i+1}\circ x_{i+1}$ for $i$ even. It follows that 
$\alpha_i=b_{i+1}+t\alpha_{i+1}$ for $i$ odd, and $t\alpha_i=b_{i+1}+\alpha_{i+1}$ for $i$ even. From this we have 
\begin{align*}
t(\alpha_{2i}-\alpha_{2i+2})&=b_{2i+1}+b_{2i+2},
\end{align*} for  $i=0,\dots,r-1$.

Thus, if $L_J$ is a direct summand of $\mathbb M(I,J)$, then $t| b_{i}+b_{i+1}$, for $i$ odd,  and we can easily find $\alpha_i$, $i=1,\dots,n$, satisfying previous equations. If only one of these divisibility conditions is not met, then $L_J$ is not a direct summand of $\MM(I,J)$. Note that if $L_J$ is not a summand of $\MM(I,J)$, it does not mean that $M$ is indecomposable (cf.\ Theorem 3.12 in \cite{BBGE}). We will study the structure of the module $\MM(I,J)$ in terms of the divisibility conditions the sums $b_i+b_{i+1}$ satisfy.

Let us now consider the general case, that is,  let  $\mathbb M(I,J)$ be the module as defined above,  when $I$ and $J$ are tightly $r$-interlacing. Write $I\setminus J$ as $\{i_1,\dots,i_r\}$ and $J\setminus I$ as $\{j_1,\dots,j_r\}$ so that 
$1\le i_1<j_1<i_2<j_2<\dots<i_r<j_r\le n$. Define 
\begin{align*}
&&&&&&x_{i_{l}}&=\begin{pmatrix} t& b_{2l-1} \\ 0 & 1 \end{pmatrix},& x_{j_{l}}&=\begin{pmatrix} 1& b_{2l} \\ 0 & t \end{pmatrix}, &&&&&\\ 
&&&&&&y_{i_{l}}&=\begin{pmatrix} 1& -b_{2l-1} \\ 0 & t \end{pmatrix}, &y_{j_{l}}&=\begin{pmatrix} t& -b_{2l} \\ 0 & 1 \end{pmatrix},&&&&&
\end{align*}
for $l=1,2,\dots, r$ (see previous figure for $n=6$).  For  $i\in  I^c \cap J^c$,  we set  $x_i=\begin{pmatrix} t & 0 \\ 0 & t \end{pmatrix}$ and $y_i=\begin{pmatrix} 1 & 0 \\ 0 & 1 \end{pmatrix}$. For  $i\in I\cap J$,  we set  $x_i=\begin{pmatrix} 1 & 0 \\ 0 & 1 \end{pmatrix}$ and $y_i=\begin{pmatrix} t & 0 \\ 0 & t \end{pmatrix}$.   Also, we assume that $\sum_{l=1}^nb_l=0$.  Note that for  $i\in  (I^c \cap J^c)\cup (I\cap J)$ we define the matrices $x_i$ and $y_i$   to be diagonal, i.e.\ we assume that the upper right corner of $x_i$ and $y_i$ is $0$ if $i\in  (I^c \cap J^c)\cup (I\cap J)$. This is because if it were not $0$, then by a suitable base change of the $V_i$, by changing the second basis element, we obtain a scalar matrix.  By construction,  $xy=yx$ and $x^k=y^{n-k}$ at all vertices, 
and $\MM(I,J)$ is free over the centre of $B_{k,n}$. Hence, the following proposition holds.

\begin{prop} 
The module $\MM(I,J)$ as constructed above is in 
${\rm CM}(B_{k,n})$. 
\end{prop}

As in the case of the profile $\{1,3,\dots, 2r-1\}\mid \{2,4,\ldots, 2r\}$,  $L_J$ is a direct summand of $\MM(I,J)$  if and only if  $t\mid b_{i}+b_{i+1}$, for all odd $i$. In order to determine the structure of the module $\MM(I,J)$ when these divisibility conditions are not fulfilled (i.e.,  at least one of the sums $b_i+b_{i+1}$ is not divisible by $t$), we determine the structure of an endomorphism of this module.   The following proposition is a generalization of Proposition 3.3  in \cite{BBL}.

For the rest of the paper, if  $t^dv=w$, for a positive integer $d$, then $t^{-d}w$ denotes $v$.

\begin{prop} \label{lm:n6-hom}
For $n\geq 6$, let $I,J$ be tightly $r$-interlacing, $I\setminus J=\{i_1,\dots,i_r\}$, and $J\setminus I=\{j_1,\dots,j_r\}$, where  $1\le i_1<j_1<i_2<j_2<\dots<i_r<j_r\le n$.  If $\varphi=( \varphi_i)_{i=1}^n\in$ {\rm End}$(\MM(I,J))$, then   
\begin{align}\label{end} \nonumber
\varphi_{j_r}&=\begin{pmatrix}a & b \\ c & d  \end{pmatrix},\\  
\varphi_{i_{l}}&=\begin{pmatrix}a+(\sum_{g=1}^{2l-1}b_g)t^{-1}c & tb+(d-a)(\sum_{g=1}^{2l-1}b_g)-(\sum_{g=1}^{2l-1}b_g)^2t^{-1}c \\ t^{-1}c& d-(\sum_{g=1}^{2l-1}b_g)t^{-1}c  \end{pmatrix},\\   \nonumber
\varphi_{j_l}&=\begin{pmatrix}a+(\sum_{g=1}^{2l}b_g)t^{-1}c & b+t^{-1}((d-a)(\sum_{g=1}^{2l}b_g)-(\sum_{g=1}^{2l}b_g)^2t^{-1}c) \\ c& d-(\sum_{g=1}^{2l}b_g)t^{-1}c  \end{pmatrix}, \\ \nonumber
\varphi_{i}&=\varphi_{i-1}, \,\, \text{for } i\in (I^c \cap J^c)\cup (I\cap J),
\end{align}
where $l=1,2,\dots, r$, with $a,b,c,d\in \CC[|t|]$, and 
\begin{align} \label{div} \nonumber
t&\mid c,\\
t&\mid (d-a)(b_1+b_2)-(b_1+b_2)^2t^{-1}c,\nonumber \\
t&\mid (d-a)(b_1+b_2+b_3+b_4)-(b_1+b_2+b_3+b_4)^2t^{-1}c,\\     \nonumber
&\vdots\\ \nonumber   
t&\mid (d-a)\sum_{i=1}^{r-1}(b_{2i-1}+b_{2i})-(\sum_{i=1}^{r-1}b_{2i-1}+b_{2i})^2t^{-1}c. \nonumber 
\end{align}
 
\end{prop}
\begin{proof} 
Let $\varphi=(\varphi_1,\dots, \varphi_n)$ be an endomorphism of $\MM(I,J)$, where 
each $\varphi_i$ is an element of $M_2(\CC[|t|])$ (matrices over the centre). We use commutativity relations $x_{i+1}\varphi_i = \varphi_{i+1}x_{i+1}$.  From $x_{i_1}x_{i_1-1}\cdots x_{j_r+1}\varphi_{j_r} = \varphi_{i_1}x_{i_1}x_{i_1-1}\cdots x_{j_r+1}$, we obtain $x_{i_1}\varphi_{j_r} = \varphi_{i_1}x_{i_1}$. Recall that $x_{i_1-1}, \dots,x_{j_r+1} $ are scalar matrices so they cancel out.    If $\varphi_{j_r}=\begin{pmatrix} a & b \\ c & d \end{pmatrix}$ and $\varphi_{i_1}=\begin{pmatrix} e & f \\ g & h \end{pmatrix}$, then  
 $t\mid c$,  $e=a+b_1t^{-1}c$,   $f=tb+(d-a)b_1-b_1^2t^{-1}c$, $g= t^{-1}c$, and  $h=d-b_1t^{-1}c$. The rest is shown in the same way.

\begin{center}
\xymatrix{
&&V_{j_r}\ar@<.8ex>[rr]^{x_{j_r+1}} \ar@<.1ex>[ddd]_{\varphi_{j_r}}
&&V_{j_r+1}\ar@<.8ex>[rr]^{x_{j_r+2}} \ar@<.8ex>[ll]^{y_{j_r+1}} 
&& \cdots  \ar@<.8ex>[rr]^{x_{i_1-1}}\ar@<.8ex>[ll]^{y_{j_r+2}} 
&& V_{i_1-1}  \ar@<.8ex>[rr]^{\begin{pmatrix} t & b_1 \\ &&0 & 1 \end{pmatrix}} \ar@<.8ex>[ll]^{y_{i_1-1}} 
&&V_{i_1}  \ar@<.8ex>[ll]^{\begin{pmatrix} 1 & -b_1 \\ &&0 & t \end{pmatrix}} \ar@<.1ex>[ddd]_{\varphi_{i_1}}\\
&&&&&&&&&&\\
&&&&&&&&&&\\
&&V_{j_r}\ar@<.8ex>[rr]^{x_{j_r+1}}
&&V_{j_r+1} \ar@<.8ex>[rr]^{x_{j_r+2}} \ar@<.8ex>[ll]^{y_{j_r+1}} 
&& \dots  \ar@<.8ex>[rr]^{x_{i_1-1}}\ar@<.8ex>[ll]^{y_{j_r+2}} 
&& V_{i_1-1}  \ar@<.8ex>[rr]^{\begin{pmatrix} t & b_1 \\ 0 & 1 \end{pmatrix}} \ar@<.8ex>[ll]^{y_{i_1-1}}
&&V_{i_1}  \ar@<.8ex>[ll]^{\begin{pmatrix} 1 & -b_1 \\ 0 & t \end{pmatrix}}\\
}
\end{center}

The only thing left to note is that if $i\in (I^c \cap J^c)\cup (I\cap J)$, then $x_i$ is a scalar matrix (either identity or $t$ times identity), so from  $x_{i}\varphi_{i-1}=\varphi_{i}x_{i}$, it follows immediately that  $\varphi_{i-1}=\varphi_{i}$.
\end{proof}

By Remark 3.4 in \cite{BBL},  if $\varphi$ is the morphism from the previous proposition, then it is sufficient to prove for a single index $i$ that $\varphi_i$ is idempotent in order to prove that $\varphi$ is idempotent. Also, note that in our computations, for $i\in (I^c \cap J^c)\cup (I\cap J),$ $x_i$ is a scalar matrix, it commutes with every other matrix and it cancels out in $x_{i}\varphi_{i-1}=\varphi_{i}x_{i}$, so it can be left out.

We now give necessary and sufficient conditions for the module $\mathbb M(I,J)$ to be indecomposable. 

\begin{theorem} \label{t6}
Let $\MM(I,J)$ be as in the previous proposition. The module $\MM(I,J)$ is indecomposable if and only if there exist odd  indices  $i_{l_1}$  and  $i_{l_2}$ such that  $t\mid b_i+b_{i+1}$, for $i_{l_1}<i<i_{l_2}$, $i$ odd, $t\nmid b_{i_{l_1}}+b_{i_{l_1}+1}$, $t\nmid b_{i_{l_{2}}}+b_{i_{l_{2}}+1}$, and $t\nmid b_{i_{l_1}}+b_{i_{l_1}+1}+b_{i_{l_{2}}}+b_{i_{l_{2}}+1}$. 
\end{theorem}
\begin{proof}
As in the proof of the previous proposition, it is sufficient to consider the case of tight $r$-interlacing, where $I=\{1,3,5,\dots,2r-1\}$ and $J=\{2,4,6, \dots, 2r\}$.  Let $i_{l_1}, i_{l_2}, \dots, i_{l_s}$ be all odd indices $i$ (in cyclic ordering) such that the sum $b_i+b_{i+1}$ is not divisible  by $t$. We assume that there is at least one such index because if $t\mid b_i+b_{i+1}$ for all odd $i$, then $\MM(I,J)$ is the direct sum $L_I\oplus L_J$.

 Let $\varphi=( \varphi_i)_{i=0}^{n-1}\in$ End$(\MM(I,J))$ be an idempotent homomorphism and assume that $\varphi_{j_r}=\begin{pmatrix}a & b \\ c & d  \end{pmatrix}$. The divisibility conditions (\ref{div}) from the previous proposition reduce to 
\begin{align} \label{div0} \nonumber
t&\mid c, \\ \nonumber  
  t&\mid (d-a)(b_{i_{l_1}}+b_{i_{l_1}+1})-(b_{i_{l_1}}+b_{i_{l_1}+1})^2t^{-1}c,  \\ \nonumber
  t&\mid (d-a)(b_{i_{l_1}}+b_{i_{l_1}+1}+b_{i_{l_2}}+b_{i_{l_2}+1})-(b_{i_{l_1}}+b_{i_{l_1}+1}+b_{i_{l_2}}+b_{i_{l_2}+1})^2t^{-1}c. \\ 
  &\vdots \\     \nonumber
t&\mid (d-a)\sum_{g=1}^{s-1}(b_{i_{l_g}}+b_{i_{l_g}+1})-\left(\sum_{g=1}^{s-1}(b_{i_{l_g}}+b_{i_{l_g}+1})\right)^2t^{-1}c.\nonumber
\end{align}
 
Without loss of generality we assume that $t\nmid b_{i_{l_1}}+b_{i_{l_1}+1}+b_{i_{l_{2}}}+b_{i_{l_{2}}+1}$. Relations (\ref{div0})  are equivalent to 
\begin{align*}t&\mid d-a-(b_{l_1}+b_{{l_1}+1})t^{-1}c,\\ 
t&\mid d-a-(b_{l_1}+b_{{l_1}+1}+b_{l_2}+b_{{l_2}+1})t^{-1}c.
\end{align*}
 
 Thus, it must hold that $$t\mid (b_{l_2}+b_{{l_2}+1})t^{-1}c,$$ and since $t\nmid b_{l_2}+b_{{l_2}+1}$, it must be that $t\mid t^{-1}c$, and subsequently that $t\mid d-a$.  

From the fact that $\varphi_{j_r}$ is idempotent and $t\mid c$ it follows that $t\mid a-a^2$ and $t\mid d-d^2$. Also, from $\varphi_{j_r}^2=\varphi_{j_r}$ it follows that either $a=d$ or $a+d=1$.  If $a=d$, then $b=c=0$ (otherwise $a=d=\frac 12$ and $\frac 14=bc$, which is not possible as $c$ is divisible by $t$), and $a=d=1$ or $a=d=0$ giving us the trivial idempotents. If $a+d=1$, then $t\mid a$ or $t\mid d$. Taking into account that $t\mid d-a$, we conclude that $t\mid a$ and $t\mid d$. This implies that $1=a+d$ is divisible by $t$, which is not true. Thus, the only idempotent homomorphisms of $\MM(I,J)$ are the trivial ones. Hence, $\MM(I,J)$ is indecomposable.

Assume now that $t\mid b_{i_{l_g}}+b_{i_{l_g}+1}+b_{i_{l_{g+1}}}+b_{i_{l_{g+1}}+1}$ for every $g<s$. 
Then the divisibility conditions (\ref{div0}) for the endomorphism $\varphi$ reduce to a single condition 
$$t\mid d-a-(b_{l_1}+b_{{l_1}+1})t^{-1}c.$$  
In order to find a non-trivial idempotent $\varphi$, we only need to find elements $a$, $b$, $c$, and $d$ in such a way that $t\mid c$ and $t\mid d-a-(b_{l_1}+b_{{l_1}+1})t^{-1}c.$ Recall that if $a=d$, then we only obtain the trivial idempotents because $t\mid c$. So it must be $a+d=1$ if we want to find a non-trivial idempotent. If we choose $a=1$, $d=0$, then $t\mid 1+(b_{l_1}+b_{{l_1}+1})t^{-1}c$. Thus, we can define $c=t(b_{l_1}+b_{{l_1}+1})^{-1}$, and $b=0$ since $a-a^2=bc$ and $c\neq 0$, to get the  idempotent:  
$$\varphi_{j_r}=\begin{pmatrix}
1& 0\\
-t(b_{l_1}+b_{{l_1}+1})^{-1}&0
\end{pmatrix}.$$ 

Since this is a non-trivial idempotent, it follows that the module $\MM(I,J)$ is decomposable. 
\end{proof}

\begin{rem}
From the previous theorem it follows that if $\MM(I,J)$ is a decomposable module, then since $\sum_{i=1}^nb_i=0$ there is an even number of odd $i$ such that $t\nmid b_{i}+b_{i+1}.$  If there were an odd number of odd $i$ such that $t\nmid b_{i}+b_{i+1}$, then for two consecutive $l_1$ and $l_2$, it would hold that $t\nmid b_{i_{l_1}}+b_{i_{l_1}+1}+b_{i_{l_{2}}}+b_{i_{l_{2}}+1}.$ Our aim is to determine all such decomposable modules, so for the rest of the paper we will assume that there is an even number of odd indices $i_{l_g}$ such that $t\nmid b_{i_{l_g}}+b_{i_{l_g}+1}$, and that $t\mid b_{i_{l_g}}+b_{i_{l_g}+1}+b_{i_{l_{g+1}}}+b_{i_{l_{g+1}}+1}$ for every $g$.
\end{rem}

\begin{corollary} If $n<6$, then there are no indecomposable rank $2$ modules in ${\rm CM}(B_{k,n})$. 
\end{corollary}

The rest of the paper is dedicated to the determination of the summands of the module  $\MM(I,J)$ in the case when this module is decomposable. It is sufficient to study the case of the filtration $\{1,3,\dots, 2r-1\}\mid \{2,4,\dots,2r\}$ when $k=r$ and $n=2r$. Then the general case of tight $r$-interlacing follows because the scalar matrices can be ignored since they do not affect any of the computations we conduct. 

Denote $I=\{1,3,\dots, 2r-1\}$ and $J=\{2,4,\dots, 2r\}.$  As before, assume that $x_i=\begin{pmatrix} t& b_{i} \\ 0 & 1 \end{pmatrix}$ for odd $i$ and $x_{i}=\begin{pmatrix} 1& b_{i} \\ 0 & t \end{pmatrix}$ for even $i$, and  that $\sum_{i=1}^{2r}b_i=0$ so that we have a module structure, which we again denote by $\MM(I,J)$. 

The dimension lattice of a given module in ${\rm CM}(B_{k,n})$ is additive on short exact sequences. 
If $\MM(I,J)$ is the direct sum $L_X\oplus L_Y$, then from the short exact sequence $$0\longrightarrow L_J\longrightarrow L_X\oplus L_Y \longrightarrow L_I \longrightarrow 0$$ follows that the dimension lattices of $L_X$ and $L_Y$ add up to the sum of the dimension lattice of $L_I$ and the dimension lattice of $L_J$.  In terms of the rims, one way to combinatorially describe possible summands $L_X$ and $L_Y$ is by the fact that the rim of $X$ has to be ``taken out'' from the lattice diagram of $L_I\oplus L_J$, i.e., of the profile $I\mid J$, in such a way that the leftover part of the lattice diagram is the rim  $Y$. 

In terms of the lattice diagram of the profile $I\mid J$ (recall that we picture the lattice diagram of $I\mid J$ by drawing the rim $J$ below the rim $I$, in such a way that $J$ is placed as far up as possible so that no part of the rim $J$ is strictly above the rim $I$), the rim $X$ corresponds to a subset of the set of the peaks of $I$, and the rim $Y$ corresponds to the complement of this set with respect to the set of peaks of $I$. To describe this in terms of the path we take in the lattice diagram of $I\mid J$ by travelling from left to right, we start from a peak of $I$ and move to the right (we either go up or down in each step). If we are at a peak of $I$ (resp.\ valley of $J$), then the next step has to be down (resp.\ up).   If we are at a peak of $J$, which is also a valley of $I$, then we have a choice of going up or down. Eventually, to finish our trip, we have to return to the peak where we started off. The rim $X$ is determined by the set of peaks of $I$ that we passed through during our trip through the lattice diagram of $I\mid J$ (by abuse of notation we say that $X$ passes through this set of peaks), and the rim $Y$ is determined by the peaks of $I$ we did not pass through.      

The first four pictures in Figure \ref{r4} correspond to the case when $X$ passes through a single peak of $I$ (and $Y$ passes through three peaks) when we travel from left to right through the lattice diagram of $I\mid J$ (or more precisely, through the rims of $I$ and $J$, see the next example). The next three pictures correspond to the  case when $X$ passes through two peaks of $I$ (and $Y$ passes through two peaks), and the last picture corresponds to the case when $X$ passes through all peaks of $I$ and $Y$ passes through none. Obviously, there is symmetry in the argument so the case when $X$ passes through one peak and $Y$ through three peaks is the same as the case when $X$ passes through three peaks and $Y$ passes through one peak.  In total, there are $2^{r-1}$ different cases, so there are $2^{r-1}$ corresponding decomposable modules.


\begin{ex}\label{48} In the case $r=4$, there are eight possible choices for $X$ and $Y$ in such a way that the sum of the dimension lattices of $X$ and $Y$ is equal to the dimension lattice of the profile $I\mid J$ . They are given in Figure \ref{r4}. Note that the set of peaks of $I$ (i.e., the set of valleys of $J$) is $\{0,2,4,6\}$ and the set of peaks of $J$ (i.e., the set of valleys of $I$) is $\{1,3,5,7\}$.
\begin{figure}[H]
\begin{center}
\subfloat[$L_{\{1, 3, 5,6\}}\oplus L_{\{2,4, 7,8\}}$]{\includegraphics[width = 6cm]{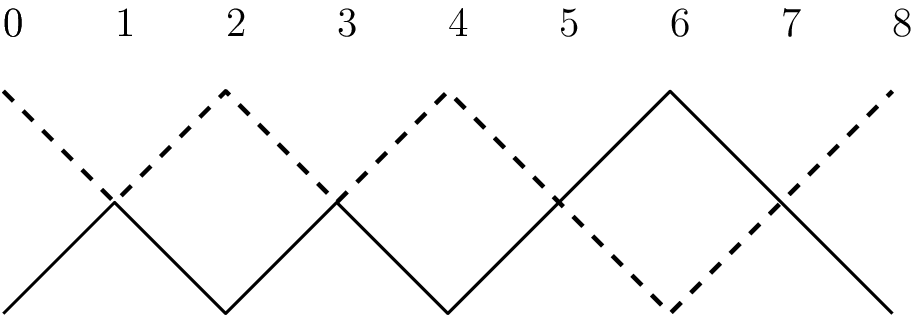}}  \quad \quad\quad\quad 
\subfloat[$L_{\{1, 2, 4, 6\}}\oplus L_{\{3,5, 7,8 \}}$]{\includegraphics[width = 6cm ]{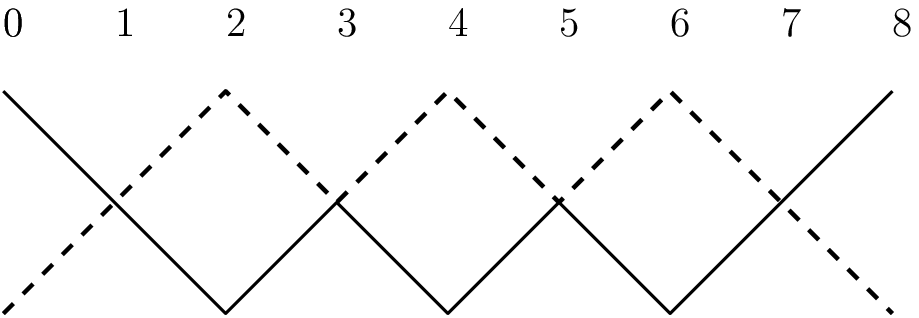}}\\
\subfloat[$L_{\{1, 2, 5 , 7\}}\oplus L_{\{3,4, 6, 8\}}$]{\includegraphics[width = 6cm ]{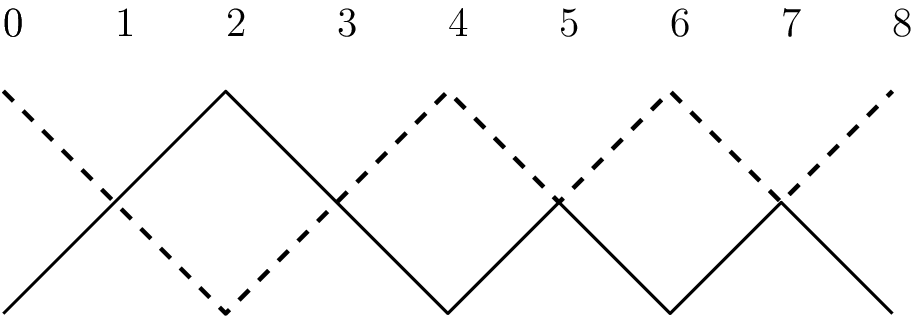}}  \quad \quad \quad \quad
\subfloat[$L_{\{1, 3, 4, 7\}}\oplus L_{\{2,5,6, 8\}}$]{\includegraphics[width = 6cm ]{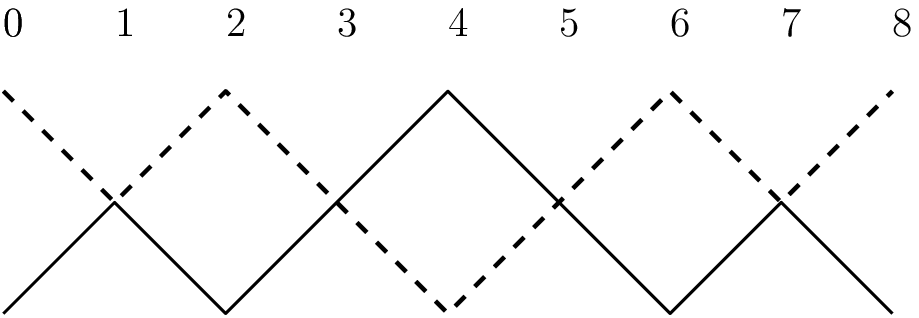}} \\
\subfloat[$L_{\{1, 3, 4,6\}}\oplus L_{\{2,5, 7,8\}}$]{\includegraphics[width = 6cm]{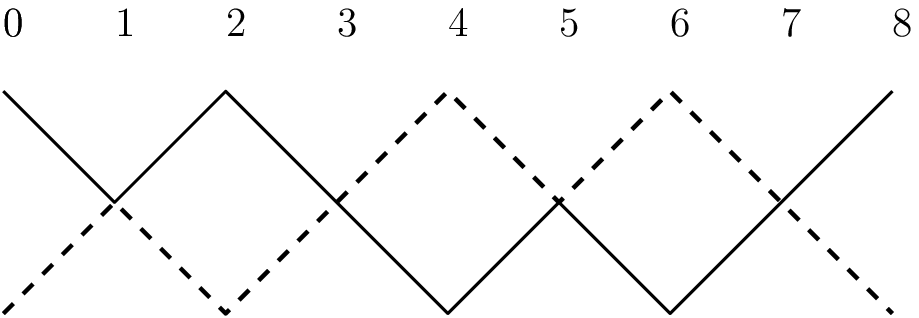}}  \quad \quad\quad\quad 
\subfloat[$L_{\{1, 2, 4, 7\}}\oplus L_{\{3,5, 6,8 \}}$]{\includegraphics[width = 6cm ]{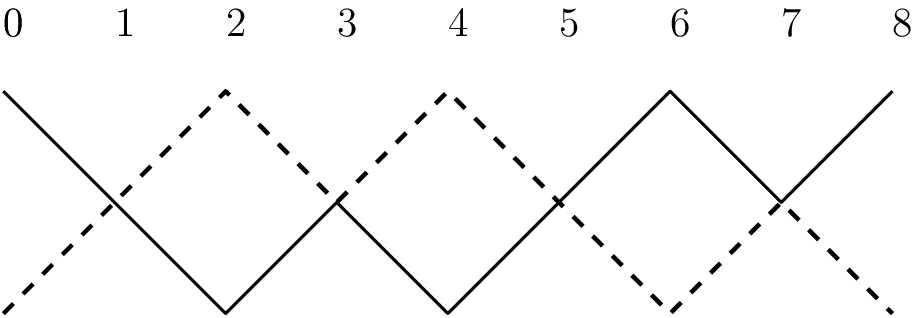}}\\
\subfloat[$L_{\{1,2,5,6\}}\oplus L_{\{3,4,7,8\}}$]{\includegraphics[width = 6cm]{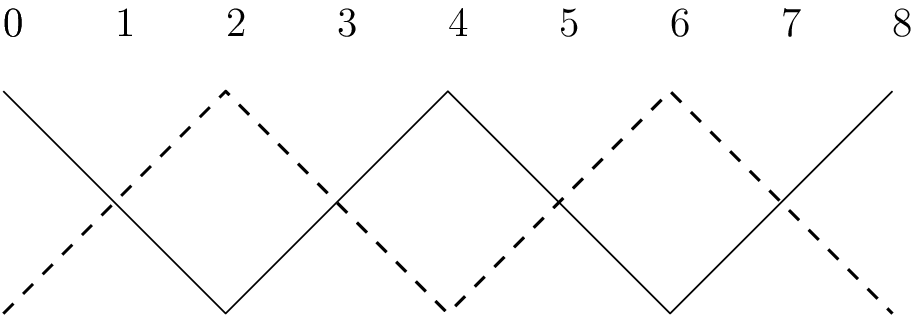}} \quad \quad \quad 
\quad
\subfloat[$L_{\{1,3,5,7\}}\oplus L_{\{2,4,6,8\}}$]{\includegraphics[width = 5.5cm]{1357-2468.eps}}
\caption{The pairs of profiles of decomposable extensions between $ L_{\{1, 3, 5,7\}}$ and $L_{\{2,4, 6,8\}}$.}
\label{r4}
\end{center}
\end{figure}
\end{ex}

Note that we only classify decomposable modules that are extensions of $L_J$ by $L_I$, not all possible extensions (cf.\ Remark 3.9 in \cite{BBL}).

For a given $X$, i.e., for a given subset of the set of peaks of $I$, and the corresponding $Y$, we give the divisibility conditions for the sums $b_i+b_{i+1}$ so that the module $\MM(I,J)$  is isomorphic to $L_X\oplus L_Y$. We denote by $X'$ (resp.\ $Y'$) the set of peaks of $I$ that corresponds to $X$ (resp.\ $Y$).

 If $X$ passes through every peak of $I$, then this is the case when $X=I$ and $Y=J$, i.e., the case of the direct sum $L_I\oplus L_J$. In terms of the divisibility conditions, this is the case when $t\mid b_i+b_{i+1}$, for every odd $i$. 
 
 Assume now that $X'$ does not contain all peaks of $I$. This means that there is a peak, say $2j$, that belongs to $X'$, such that the next peak $2j+2$ belongs to $Y'$. Then $b_{2j+1}+b_{2j+2}$ is not divisible by $t$. If it were divisible by $t$, then $2j+2$ would belong to $X'$ as we explain below.

If the current peak, say $2j$, belongs to $X'$, i.e., we are at the peak $2j$, then if the next peak $2j+2$ belongs to $X'$, then $2j+1\in X$, $2j+1\notin Y$, $2j+2\notin X$, and $2j+2\in Y$.  In this situation we are moving from a peak to another peak by going down and then up. Note that we pass through two vertices each time we move from one peak to another.   

If the current peak $2j$ does not belong to $X'$, i.e., we are at the valley $2j$, then if the next peak $2j+2$ belongs to $X'$, then $2j+1, 2j+2\notin X$, and $2j+1, 2j+2\in Y$. In this situation we are moving from a valley to a peak by going up and up.

If the current peak $2j$ belongs to $X'$, i.e., we are at the peak $2j$, then if the next peak $2j+2$ does not belong to $X'$, then  $2j+1, 2j+2\in X$, and $2j+1, 2j+2\notin Y$.  In this situation we are moving from a peak to a valley  by going down and down.

If the current peak $2j$ does not belong to $X'$, i.e., we are at the valley $2j$, then if the next peak $2j+2$ does not belong to $X'$, then $2j+1\notin X$, $2j+1\in Y$, $2j+2\in X$, and $2j+2\notin Y$. In this situation we are moving from a valley to another valley by going up and then down.

Recall that there has to be an even number of steps where we go from a valley to a peak or from a peak to a valley so that we can come back up to the point where we started. 
      
\begin{theorem} \label{paths}
Let $X'$ be a subset of the set of peaks of $I$, $\{0,2,4,\dots,2r-2\}$, $Y'$ its complement,  and $X$ and $Y$ corresponding $k$-subsets of $[n]$. Also, assume that $1\leq |X'|<r$. Starting from a peak in $X'\setminus Y'$, and moving to the right, define sums $b_i+b_{i+1}$ so that the following conditions hold:
\begin{enumerate}
\item if the current peak $2j$ belongs to $X'$, then if the next peak $2j+2$ belongs to $X'$, then $t\mid b_{2j+1}+b_{2j+2}$,
\item if the current peak $2j$ does not belong to $X'$, then if the next peak $2j+2$ belongs to $X'$, then $t\nmid b_{2j+1}+b_{2j+2}$,
\item if the current peak $2j$ belongs to $X'$, then if the next peak $2j+2$ does not belong to $X'$, then $t\nmid b_{2j+1}+b_{2j+2}$,
\item if the current peak $2j$ does not belong to $X'$, then if the next peak $2j+2$ does not belong to $X'$, then $t\mid b_{2j+1}+b_{2j+2}$.
\end{enumerate}

Additionally, we assume that $t\mid b_{i_1}+b_{i_1+1}+b_{i_2}+b_{i_2+1}$ for every two consecutive odd indices $i_1$ and $i_2$ such that $t\nmid b_{i_l}+b_{i_{l+1}}$, $l=1,2$. Then the module $\mathbb{M}(I,J)$ is isomorphic to the direct sum $L_X\oplus L_Y$. 
\end{theorem}
\begin{proof}
By Theorem \ref{t6}, the module is decomposable. We start our path at a peak from $X'\setminus Y'$. Assume without loss of generality that this peak is 0 and that the next peak 2 does not belong to $X'$. This means that $t\nmid b_{1}+b_{2}$. Define an idempotent $\varphi_{0}=\begin{pmatrix}
1& 0\\
-t(b_{1}+b_{2})^{-1}&0
\end{pmatrix}.$
 Its orthogonal complement is the idempotent
 $
 \tilde{\varphi}_{0}=\begin{pmatrix}
0& 0\\
t(b_{1}+b_{2})^{-1}&1
\end{pmatrix}.$  From (\ref{end}) we easily compute other idempotents $\varphi_i$.  If we denote by  $B_l$  the sum $\sum_{i=1}^lb_i,$ then  for odd indices we get
$$
\varphi_{2j+1}=\begin{pmatrix}1-B_{2j+1}(b_1+b_{2})^{-1} & -B_{2j+1}+B_{2j+1}^2(b_1+b_{2})^{-1} \\ -(b_1+b_{2})^{-1}& -B_{2j+1}(b_1+b_{2})^{-1}  \end{pmatrix},$$ 
$$\tilde{\varphi}_{2j+1}=\begin{pmatrix}B_{2j+1}(b_1+b_{2})^{-1} & B_{2j+1}(1-B_{2j+1}(b_1+b_{2})^{-1}) \\ (b_1+b_{2})^{-1}& 1-B_{2j+1}(b_1+b_{2})^{-1}  \end{pmatrix},
$$
and for even indices 
$$
\varphi_{2j}=\begin{pmatrix}1-B_{2j}(b_1+b_{2})^{-1} & -t^{-1}B_{2j}(1-B_{2j}(b_1+b_{2})^{-1}) \\ -t(b_1+b_{2})^{-1}& B_{2j}(b_1+b_{2})^{-1}  \end{pmatrix},$$ 
$$\tilde{\varphi}_{2j}=\begin{pmatrix}B_{2j}(b_1+b_{2})^{-1} & t^{-1}B_{2j}(1-B_{2j}(b_1+b_{2})^{-1}) \\ t(b_1+b_2)^{-1}& 1-B_{2j}(b_1+b_{2})^{-1}\end{pmatrix}.
$$

Let $v_i$ (resp.\ $w_i$) be the eigenvector of  $\varphi_{i}$ (resp.\ $\tilde{\varphi}_{i}$) corresponding to the eigenvalue $1$. The vectors $w_i$ (resp.\ $v_i$) form a basis for $L_X$ (resp.\ $L_Y$). We compute directly these eigenvectors. For an odd index we have
$$
w_{2j+1}=\begin{pmatrix} B_{2j+1}\\ 1
\end{pmatrix}, \quad \quad \quad v_{2j+1}=\begin{pmatrix} 1-B_{2j+1}(b_1+b_2)^{-1}\\ -(b_1+b_2)^{-1}
\end{pmatrix}. 
$$  
 
Since $t\mid b_{i_1}+b_{i_1+1}+b_{i_2}+b_{i_2+1}$, for every two consecutive odd indices $i_1$ and $i_2$ such that $t\nmid b_{i_l}+b_{i_l+1}$,  $l=1,2$, when computing $w_{2j}$ and $v_{2j}$ we have to distinguish between the following cases.  Let $g$ be the number of indices in the set $[1,j]$ such that $t\nmid b_{2j-1}+b_{2j}$. 
If $g$ is even (resp.\ odd), then $t\mid B_{2j}$ (resp.\ $t\nmid B_{2j}$).
Therefore, if $g$ is even, i.e., if $B_{2j}$ is divisible by $t$, then 
$$
w_{2j}=\begin{pmatrix} t^{-1}B_{2j}\\ 1
\end{pmatrix}, \quad \quad \quad v_{2j}=\begin{pmatrix} 1-B_{2j}(b_1+b_2)^{-1}\\ -t(b_1+b_2)^{-1}
\end{pmatrix}. 
$$  
If $g$ is odd, i.e., if $B_{2j}$ is not divisible by t (more precisely, $B_{2j}=b_1+b_2+tz$, for some $z$), then 
$$
w_{2j}=\begin{pmatrix} B_{2j}\\t  
\end{pmatrix}, \quad \quad \quad v_{2j}=\begin{pmatrix} t^{-1}(1-B_{2j}(b_1+b_2)^{-1})\\ -(b_1+b_2)^{-1}
\end{pmatrix}. 
$$  
Combinatorially, $g$ is even (resp.\ odd) if and only if we are positioned at a peak (resp.\ valley) $2j$ after $(2j)$th step. This follows from the fact that $t\nmid b_{2j-1}+b_{2j}$ means that we are moving either from a peak to a valley, or from a valley to a peak. Since we started from a peak, if we are currently at a peak $2j$, then this means that we had an even number of the moves that correspond to the sums $b_{2i-1}+b_{2i}$ that are not divisible by $t$.

Consider the eigenvectors  $v_0=[1\,\, ,\,\, -t(b_1+b_2)^{-1}]^t$, $w_0=[0\,\, ,\,\, 1]^t$ for $\varphi_0$ and its orthogonal complement. Then  $x_1w_0=w_1$ and $x_2w_1=w_2$, so $1,2\in X$. Also, $x_1v_0=tv_1$ and $x_2v_1=tv_2$, so $1,2\notin Y$.  We continue by moving to the right and consider the four cases from the statement of the theorem.

 Case 1:  If the current peak $2j$ belongs to $X'$, then if the next peak $2j+2$ belongs to $X'$, then $t\mid b_{2j+1}+b_{2j+2}$.  In this situation we are moving from a peak to another peak by going down and then up. Here, $t\mid B_{2j}$ and $t\mid B_{2j+2}$.  Since   $w_{2j}=\begin{pmatrix} t^{-1}B_{2j}\\1 \end{pmatrix},$ $w_{2j+1}=\begin{pmatrix} B_{2j+1}\\ 1
\end{pmatrix},$ and $w_{2j+2}=\begin{pmatrix} t^{-1}B_{2j+2}\\1 \end{pmatrix},$ it follows that $x_{2j+1}w_{2j}=w_{2j+1}$ and $x_{2j+2}w_{2j+1}=tw_{2j+2}$. Therefore, $2j+1\in X$ and $2j+2\notin X$. Analogously,  $x_{2j+1}v_{2j}=tv_{2j+1}$ and $x_{2j+2}v_{2j+1}=v_{2j+2}$. Therefore, $2j+1\notin Y$ and $2j+2\in Y$. 
 
Case 2: If the current peak $2j$ does not belong to $X'$, then if the next peak $2j+2$ belongs to $X'$, then $t\nmid b_{2j+1}+b_{2j+2}$ and we are moving from a valley to a peak. Here, $t\nmid B_{2j}$ and $t\mid B_{2j+2}$. In this case $w_{2j}=\begin{pmatrix} B_{2j}\\t \end{pmatrix},$ $w_{2j+1}=\begin{pmatrix} B_{2j+1}\\ 1\end{pmatrix},$ and $w_{2j+2}=\begin{pmatrix} t^{-1}B_{2j+2}\\1 \end{pmatrix}.$ It follows that $x_{2j+1}w_{2j}=tw_{2j+1}$ and $x_{2j+2}w_{2j+1}=tw_{2j+2}$. Therefore, $2j+1\notin X$ and $2j+2\notin X$. Analogously,  $x_{2j+1}v_{2j}=v_{2j+1}$ and $x_{2j+2}v_{2j+1}=v_{2j+2}$. Therefore, $2j+1\in Y$ and $2j+2\in Y$. 

Case 3:  If the current peak $2j$ belongs to $X'$, then if the next peak $2j+2$ does not belong to $X'$, then $t\nmid b_{2j+1}+b_{2j+2}$ and we move from a peak to a valley. Here, $t\mid B_{2j}$ and $t\nmid B_{2j+2}$. In this case $w_{2j}=\begin{pmatrix} t^{-1}B_{2j}\\1 \end{pmatrix},$ $w_{2j+1}=\begin{pmatrix} B_{2j+1}\\ 1\end{pmatrix},$ and $w_{2j+2}=\begin{pmatrix} B_{2j+2}\\t \end{pmatrix}.$ It follows that $x_{2j+1}w_{2j}=w_{2j+1}$ and $x_{2j+2}w_{2j+1}=w_{2j+2}$. Therefore, $2j+1\in X$ and $2j+2\in X$. Analogously,  $x_{2j+1}v_{2j}=tv_{2j+1}$ and $x_{2j+2}v_{2j+1}=tv_{2j+2}$. Therefore, $2j+1\notin Y$ and $2j+2\notin Y$. 

Case 4: If the current peak $2j$ does not belong to $X'$, then if the next peak $2j+2$ does not belong to $X'$, then $t\mid b_{2j+1}+b_{2j+2}$ and we move from a valley to a valley.  Here, $t\nmid B_{2j}$ and $t\nmid B_{2j+2}$. In this case $w_{2j}=\begin{pmatrix} B_{2j}\\t \end{pmatrix},$ $w_{2j+1}=\begin{pmatrix} B_{2j+1}\\ 1\end{pmatrix},$ and $w_{2j+2}=\begin{pmatrix} B_{2j+2}\\t \end{pmatrix}.$ It follows that $x_{2j+1}w_{2j}=tw_{2j+1}$ and $x_{2j+2}w_{2j+1}=w_{2j+2}$. Therefore, $2j+1\notin X$ and $2j+2\in X$. Analogously,  $x_{2j+1}v_{2j}=v_{2j+1}$ and $x_{2j+2}v_{2j+1}=tv_{2j+2}$. Therefore, $2j+1\in Y$ and $2j+2\notin Y$.  
\end{proof}
\begin{ex} Consider Figure (f) from Example \ref{48} (see Figure \ref{48}).
\begin{figure}[H]
\begin{center}
{\includegraphics[width = 8cm ]{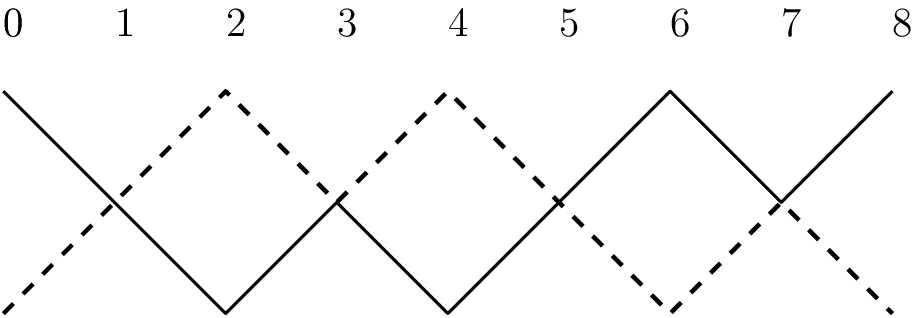}}\\
\caption{A pair of profiles for $L_{\{1, 2, 4, 7\}}\oplus L_{\{3,5, 6,8 \}}$}
\label{48}
\end{center}
\end{figure}
Here, $X=\{1,2,4,7\}$, $X'=\{0,6\}$, $Y=\{3,5,6,8\}$, and $Y'=\{2,4\}$.  Define $b_i$, $i=1,\dots, 8$, as follows. We start at the peak 0, and we travel to the right by going through two points at each step. We first reach valley 2 by going down and down. Here, $t\nmid b_1+b_2$ because of the third condition from the previous theorem. Then we reach valley 4 by going up and down. Here, $t\mid b_3+b_4$ because of the fourth condition from the previous theorem. Next, we reach peak 6 by going up and up. By the second condition from the previous theorem, it must be $t\nmid b_5+b_6$. Finally, we come back to the starting peak 0 by going down and then up. As stated in the first condition of the previous theorem, we have that $t\mid b_7+b_8$. Therefore, if $t\nmid b_1+b_2$, $t\mid b_3+b_4$, $t\nmid b_5+b_6$, and $t\mid b_7+b_8$, then the module $\mathbb M(I,J)$ is isomorphic to  $L_{\{1, 2, 4, 7\}}\oplus L_{\{3,5, 6,8 \}}$. Note that we also have to make sure that  $\sum_{i=1}^8b_i=0$. For example, we can set $b_1+b_2=-(b_5+b_6)=1$ and $b_3+b_4=-(b_7+b_8)=t$.
\end{ex}

\begin{rem}It is not too difficult to generalize Theorem \ref{paths}, by taking analogous paths in the lattice diagram, to the general case when the layers of the profile $I\mid J$ are $r_1$-interlacing for some $r_1\ge 3$, and the profile $I\mid J$ has $r$ boxes, with poset $1^r\mid 2$ (we refer the reader to \cite{BBL} for details on the notion of a box, the poset of a profile, and a branching point of a profile). In the next  example we demonstrate how the decomposable extensions look like for a rank 2 module in the  tame case  $(4,8)$. The path is analogous to the path in the tight interlacing case, at each branching point (a point where the rims meet) we have an option to either go up or down. This path uniquely determines the summands $L_X$ and $L_Y$.
\end{rem}

\begin{ex} \cite[Example 4.5]{BBL}
To construct decomposable modules with the profile $\{2,4,7,8\}\mid \{1,3,5,6\}$  we define $x_i=\begin{pmatrix} t& b_{i} \\ 0 & 1 \end{pmatrix}$, 
$y_i=\begin{pmatrix} 1& -b_{i} \\ 0 & t \end{pmatrix},$ for $i=2,4,7,8$ and 
$x_{i}=\begin{pmatrix} 1& b_{i} \\ 0 & t \end{pmatrix},$ $y_i=\begin{pmatrix} t& -b_{i} \\ 0 & 1 \end{pmatrix},$ for  $i=1,3,5,6$, and  assume that 
$b_1+b_2+b_3+b_4+b_5+b_8+t(b_6+b_7)=0.$ Denote this module again by $\MM(I,J)$. It is easily seen that $L_J$ is a summand of $\MM(I,J)$ if and only if $t\mid b_8+b_1$, $t\mid b_2+b_3$, and $t\mid b_4+b_5$. The module $\MM(I,J)$ is indecomposable if and only if  $t\nmid b_2+b_3$, $t\nmid b_4+b_5$, $t\nmid b_8+b_1$. If   $t\nmid b_2+b_3$, $t\mid b_4+b_5$, $t\nmid b_8+b_1$, then  $\MM(I,J)$ is isomorphic to $L_{\{2,3,5,6\}}\oplus L_{\{1,4,7,8\}}$. If   $t\mid b_2+b_3$, $t\nmid b_4+b_5$, $t\nmid b_8+b_1$, then  $\MM(I,J)$ is isomorphic to  $L_{\{2,4,5,6\}}\oplus L_{\{1,3,7,8\}}$.  If  $t\nmid b_2+b_3$, $t\nmid b_4+b_5$, $t\mid b_8+b_1$, then  $\MM(I,J)$ is isomorpic to  $L_{\{2,3,7,8\}}\oplus L_{\{1,4,5,6\}}$.

Thus, there are four different decomposable modules appearing as the middle term in a short exact sequence that has  $L_I$ (as a quotient) and $L_J$ (as a submodule) as end terms. The {pairs of profiles} of the four modules that appear in the middle in these short exact sequences can be pictured as follows.
\begin{figure}[H]
\begin{center}
\subfloat[$L_{\{1, 3, 5,6\}}\oplus L_{\{2,4, 7,8\}}$]{\includegraphics[width = 6cm]{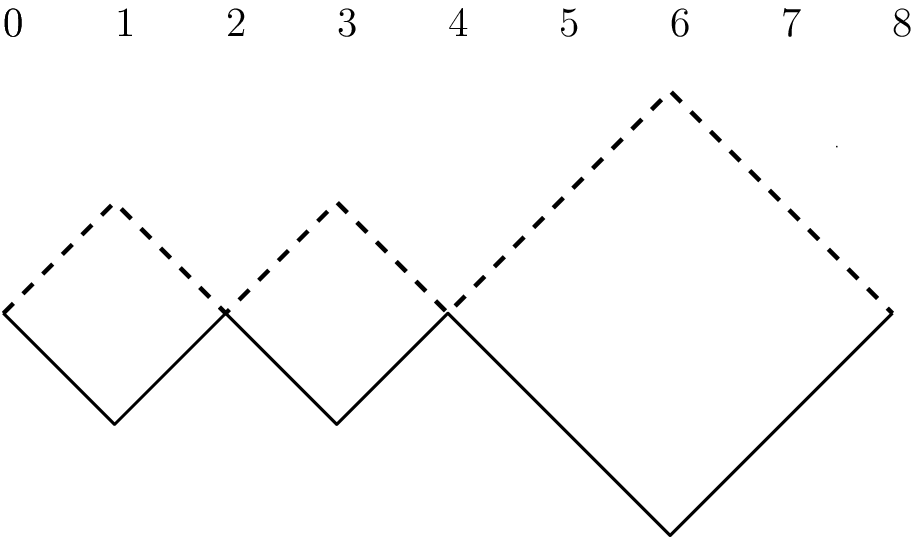}}  \quad \quad \quad \quad
\subfloat[$L_{\{1, 4, 7,8\}}\oplus L_{\{2,3,5, 6\}}$]{\includegraphics[width = 6cm ]{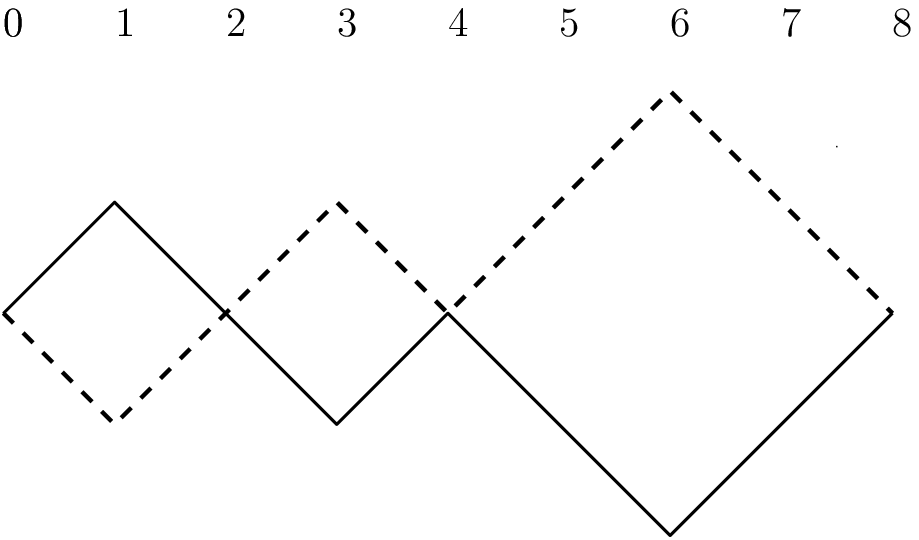}}\\
\subfloat[$L_{\{1, 3, 7,8\}}\oplus L_{\{2,4, 5, 6\}}$]{\includegraphics[width = 6cm ]{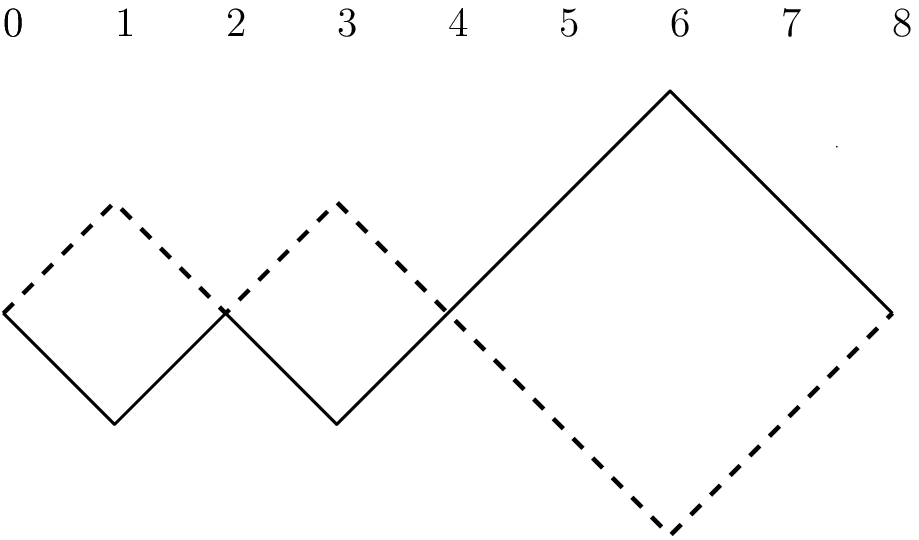}}  \quad \quad \quad \quad
\subfloat[$L_{\{2,3,7,8\}} \oplus L_{\{1, 4, 5,6\}} $]{\includegraphics[width = 6cm ]{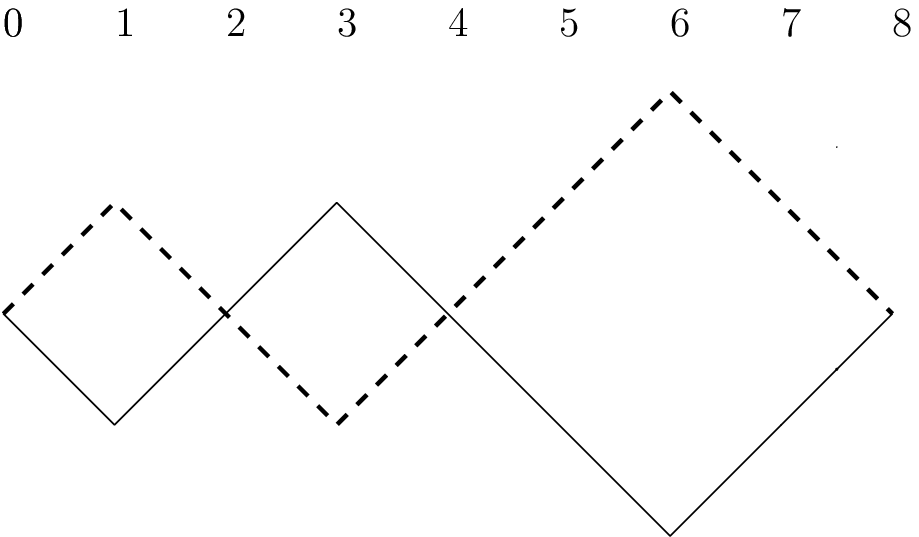}} 
\caption{{The pairs of} profiles of decomposable extensions between $ L_{\{1, 3, 5,6\}}$ and $L_{\{2,4, 7,8\}}$.}
\end{center}
\end{figure}
\end{ex}



\end{document}